\documentclass[12pt,letterpaper]{article}
\usepackage[utf8]{inputenc}
\usepackage[T1]{fontenc}
\usepackage{amsmath}
\usepackage{amsfonts}
\usepackage{amssymb}
\usepackage{amsthm}
\usepackage{graphicx}
\usepackage[width=8.50in, height=11.00in, left=1.0in, right=1.0in, top=1.0in, bottom=1.0in]{geometry}
\usepackage[colorlinks,linkcolor=blue,citecolor=red]{hyperref}
\usepackage{xcolor}
\usepackage{comment}

\newcounter{tbox}

\newtheorem{theorem}{Theorem}[section]
\newtheorem{lemma}[theorem]{Lemma}
\newtheorem{corollary}[theorem]{Corollary}

\newtheorem*{lemma*}{Lemma}

\title{Rooted $C_5$-Minors}
\author{Xiying Du\footnote{School of Mathematics,
        Georgia Institute of Technology}, Yanjia Li\footnote{School of Mathematics,
        Georgia Institute of Technology},  Xingxing Yu\footnote{School of Mathematics,
        Georgia Institute of Technology; partially supported by NSF grant DMS--2348702}}

\begin{document}

	\maketitle

\begin{abstract}
Let $G$ be a graph and $x_1, x_2, \ldots, x_k$ be distinct vertices of $G$. We say $(G,x_1x_2\ldots x_k)$ has a $C_k$-minor or $G$ has a $C_k$-minor rooted at $x_1x_2\ldots x_k$, if there exist pairwise disjoint sets $X_1, X_2, \ldots, X_k\subseteq V(G)$, such that for all $i\in [k]$, $G[X_i]$ is connected, $x_i\in X_i$,  and $G$ has an edge between $X_i$ and $X_{i+1}$, where $X_{k+1}=X_k$. When $k=3$ it is easy to determine when $(G,x_1x_2x_3)$ contains a $C_3$-minor. For $k=4$, Robertson, Seymour and Thomas gave a characterization of $(G,x_1x_2x_3x_4)$ with no $C_4$-minor, which, in particular, implies that such $G$ has connectivity at most 5. In this paper, we apply a method of Thomas and Wollan to prove a result, which implies that if $G$ is $10$-connected then,  for all distinct vertices $x_1,x_2,x_3,x_4,x_5$ of $G$, $(G,x_1x_2x_3x_4x_5)$ has a $C_5$-minor.
\end{abstract}    

\newpage
    
\section{Introduction}
    For a graph $G$ and sets $X,Y\subseteq V(G)$,  $e_G(X,Y)$ denotes the number of edges of $G$ with one vertex in $X$ and one vertex in $Y$, $\rho_G(X)$ denotes the number of edges of $G$ with at least one vertex in $X$,  $G[X]$ denotes the subgraph of $G$ induced by $X$, and $N_G(X)$ denotes the neighborhood of $X$.  When $X=\{x\}$ and $Y=\{y\}$, we will write $N_G(x)$ for $N_G(\{x\})$ and $e_G(x,y)$ for $e_G(\{x\},\{y\})$.  When no confusion arises, we may drop the subscript $G$. 
    
    We say that a graph $G$ contains another graph $H$ as a \emph{minor}, if $H$ can be obtained from a subgraph of $G$  by contracting edges. Equivalently, $H$ is a minor of $G$ if there exist  pairwise disjoint vertex sets $X_u\subseteq V(G)$ for $u\in V(H)$, such that $G[X_u]$ is connected for all $u\in V(H)$, and $e_G(X_u,X_v)>0$ for all $uv\in E(H)$. There has been extensive work on structure of graphs with a fixed graph excluded as a minor, including those motivated by Hadwiger's conjecture (which states that if a graph $G$ does not contain $K_{t+1}$ as a minor then $\chi(G)\le t$), see \cite{Se16}.

    When studying structure of graphs with excluded minors, it is often helpful to know the whereabouts of smaller minors. Let $G$ and $H$ be graphs, and let   $X\subseteq V(G)$ with $|X|=|V(H)|$.  We say that $(G,X)$ has an \emph{$H$-minor} or $G$ has an \emph{$H$-minor rooted on $X$} if there exist pairwise disjoint sets $X_u$ for $u\in V(H)$ such that all $G[X_u]$ are connected and, for some bijection $\sigma: X\rightarrow V(H)$, $x\in X_{\sigma(x)}$ for every $x\in X$, and $e_G(X_{\sigma(u)},X_{\sigma(v)})>0$ for all $\sigma(u)\sigma(v)\in E(H)$. When $|X|\le 3$,  deciding whether $(G,X)$ contains an $H$-minor is straightforward.  For the case $|X|=4$, Robertson, Seymour and Thomas \cite{HadwigerK6} (also see Fabila-Monroy and Wood \cite{K4}) characterized those $(G,X)$ without rooted $K_4$-minor (as well as rooted $H$-minor for other 4-vertex graphs), on their way to prove the Hadwiger conjecture for $t=5$.  Wollan \cite{H-bound}  showed that for any graph $H$, there exists a positive integer $f(H)$ such that, if $G$ is an $f(H)$-connected graph then, for every $X\subseteq V(G)$ with $|X|=|V(H)|$, $(G,X)$ has an $H$-minor. 
    
    We are interested in the rooted $H$-minor problem for graphs $H$ with $|V(H)|\ge 5$.  In fact, we consider an \emph{ordered} version which is related to the linkage problem for graphs. Let $G$ and $H$ be graphs with $V(H)=\{v_1,\ldots, v_h\}$, and let $X=\{x_1,\ldots, x_h\}\subseteq V(G)$. We say that $(G,x_1x_2\ldots x_h)$ has an  $H$-minor if there exist pairwise disjoint sets $X_i\subseteq V(G)$, $i\in [h]$,  such that $G[X_i]$ is connected and $x_i\in X_i$ for all $i\in [h]$, and for some automorphism $\sigma$ of $H$, $e_G(X_i,X_j)>0$ for all $i,j\in [h]$ with $\sigma(v_i)\sigma(v_j)\in E(H)$.  When $H\cong kK_2$ (i.e., $h=2k$ and $E(H)=\{v_{2i-1}v_{2i}: i\in [k]\}$), $(G,x_1\ldots x_{2k})$ contains an $H$-minor is the same as $G$ contains $k$ vertex disjoint paths from $x_{2i-1}$ to $x_{2i}$, $i\in [k]$, respectively. (Such a collection of $k$ disjoint paths with ends specified is usually called a \emph{$k$-linkage}.)
     We say that a graph $G$ is \emph{$k$-linked} if, for every sequence $x_1x_2\ldots x_{2k}$ of $2k$ distinct vertices of $G$, $(G,x_1x_2\ldots x_{2k})$ has a $kK_2$-minor.   
    
     Larman and Mani \cite{klink-existLM} and Jung \cite{klink-existJung} proved independently that for each positive integer $k$ there exists a smallest integer $f(k)$ such that every $f(k)$-connected graph is $k$-linked. Their upper bound for $f(k)$ has been improved over the time, see \cite{klink-RS, klink-K, klink-T, klink-BT, klink-KKY}. The current best bound is $f(k)\le 10k$ by Thomas and Wollan \cite{klink-TW}, where they proved a stronger result using a relaxed notion of connectivity which should be useful for studying problems with connectivity constraints. (For example, Thomas and Wollan \cite{3link-TW} used the same method to show that $f(3)\le 10$.)  

     When $H=C_k$, a $k$-cycle, $(G,x_1\ldots x_k)$ has a $C_k$-minor if there exist pairwise disjoint sets $X_1, \ldots, X_k\subseteq V(G)$, such that for $i\in [k]$, $G[X_i]$ is connected, $x_i\in X_i$, and $e(X_i,X_{i+1})>0$ (where $X_{k+1}=X_1$). 
     For a graph $G$ and a subset $X\subseteq V(G)$, we say that $(G,X)$ is \emph{cycle-linked} if $|X|\ge 3$ and for every permutation $X'$ of  $X$, $(G,X')$ has a $C_{|X|}$-minor; or if $|X|\in [2]$ and $G$ contains a path between the vertices of $X$. For any positive integer $k$, we say that a graph $G$ is \emph{$C_k$-minor-linked} if $(G,X)$ is cycle-linked for all $X\subseteq V(G)$ with $|X|=k$. 
     Note that a straightforward application of the bound $f(k)\le 10k$ shows that $50$-connected graphs are $C_5$-minor-linked. In this paper, we improve 50 to 10, using the method of Thomas and Wollan \cite{klink-TW}.

     \begin{theorem}\label{thm:10-conn}
		 $10$-connected graphs are $C_5$-minor-linked.
	\end{theorem}

	We will prove a stronger result from which we can derive Theorem~\ref{thm:10-conn}. To state and prove that result, we use the weaker version of connectivity introduced by Thomas and Wollan in \cite{klink-TW}.  
	Given a graph $G$ and a set $X\subseteq V(G)$, an ordered pair $(A,B)$ of subsets of $V(G)$ is a \emph{separation} of $(G,X)$ if $A\cup B=V(G)$, $X\subseteq A$, and $e(A\setminus B,B\setminus A)=0$; and we say that  $|A\cap B|$ is the \emph{order} of the separation $(A,B)$. 
	A separation $(A,B)$ of $(G,X)$ is \emph{rigid} if $B\setminus A\ne\emptyset$ and $(G[B], A\cap B)$ is cycle-linked. For a positive real number $\lambda$, we say that $(G,X)$ is \emph{$\lambda$-massed} if 
	\begin{itemize}
		\item[\textrm{(M1)}] \label{M1}  $\rho(V(G)\setminus X)> \lambda |V(G)\setminus X|$, and 
		\item[\textrm{(M2)}] \label{M2} $\rho(B\setminus A)\leq \lambda |B\setminus A|$ for every separation $(A,B)$ of $(G,X)$ of order less than $|X|$.
	\end{itemize}
    
	 For a graph $G$ and $S\subseteq V(G)$, $G-S$ denote the graph obtained from  $G$ by deleting vertices in $S$ and all edges incident with a vertex in $S$. The main result of this paper is the following.	
	 
	\begin{theorem} \label{main}
		Let $G$ be a graph and let $X\subseteq V(G)$, such that $|X|\le5$ and $(G,X)$ is $5$-massed. Suppose $(G,X)$ is not cycle-linked. Then $|X|=5$, $\rho(V(G)\setminus X)=5|V(G)\setminus X|+1$,  the vertices in $X$ may be labeled as $x_1,\ldots, x_5$ such that $x_ix_{i+1}\notin E(G)$ for $i\in [5]$ (with $x_6=x_1$), and        there exist $a,b\in V(G)\setminus X$ such that
        \begin{itemize}
            \item [(1)] $ab\in E(G)$ and $a,b\in N(x_i)$ for all $i\in [5]$, and
            \item [(2)] if $C$ is a component of $G-(X\cup \{a,b\})$ then $\rho_G(V(C))=5|V(C)|$ and $N_G(V(C))\subseteq \{a,b,x_i,x_{i+2}\}$ for some $i\in [5]$ (with $x_6=x_1$ and $x_7=x_2$). 

        \end{itemize}
	\end{theorem}    

\noindent\textit{Proof of Theorem \ref{thm:10-conn} assuming Theorem \ref{main}}. Let $G$ be a 10-connected graph and $X\subseteq V(G)$ with $|X|=5$.
Since $G$ is 10-connected, we have $\delta(G)\ge 10$; so $|E(G)|\ge 10|V(G)|/2=5|V(G)|$, and hence \[
\rho(V(G)\setminus X)=|E(G)|-|E(G[X])|\ge 5|V(G)|-10>5(|V(G)|-5)=5|V(G)\setminus X|;
\]
so $(G,X)$ satisfies \hyperref[M1]{(M1)}. Moreover, \hyperref[M2]{(M2)} holds vacuously as $(G,X)$ admits no separation of order less than 5. 
Suppose $(G,X)$ is not cycle-linked, then, by Theorem \ref{main}, there exists $a,b\in V(G)\setminus X$ and labeling $x_1,\ldots,x_5$ of vertices in $X$ such that (1) and (2) holds.
Note that  $G-(X\cup \{a,b\})$ has at least one component $C$ as $\delta(G)\ge 10$. By (2),  $|N_G(V(C))|\le 4$, a contradiction as $G$ is 10-connected. Thus, $(G,X)$ must be cycle-linked. \qed

\medskip
   
	 The proof of Theorem~\ref{main} is divided into two stages. In the first stage, we show that a minor minimal counterexample must contain a dense subgraph with at most 10 vertices. This is done in Section \ref{sec2}. In the second stage, we use such small dense subgraph to derive a contradiction by finding the desired $C_5$-minor, see Section \ref{sec3}. 
		
\section{Structure of minimal counterexamples}\label{sec2}  
	Suppose Theorem \ref{main} does not hold. Then there exists a graph $G$ and a set $X\subseteq V(G)$ such that:
	\begin{enumerate}
	\item[(a)] $|X|\le 5$, and $(G,X)$ is 5-massed, i.e., it satisfies (M1) and (M2);
        \item[(b)] $(G,X)$ is not cycle-linked;
        \item[(c)] the conclusion of Theorem~\ref{main} fails; 
		\item[(d)] subject to (a), (b), and (c), $|V(G)|+|E(G)|$ is minimal.
	\end{enumerate}

    We now proceed to show that $G$ contains a small dense subgraph. 
    
  \begin{lemma}\label{exception}
         $|X|<5$, or $\left|\bigcap_{x\in X} N(x)\right|\le 1$.
    \end{lemma}
    \begin{proof}
        For, suppose $|X|=5$ and let $a,b\in \bigcap_{x\in X}N(x)$ be distintct. Let $x_1,x_2,x_3,x_4,x_5$ be an ordering of vertices in $X$, such that 
        $(G,x_1x_2x_3x_4x_5)$ has no $C_5$-minor.
        
        Note that, for each $i\in[5]$, $G- \{x_{i+2},x_{i+3},x_{i+4}\}$ contains no $x_i$-$x_{i+1}$ path; since if $P$ is such a path then  $V(P-x_{i+1}), \{x_{i+1}\}, \{x_{i+2}, a\}, \{x_{i+3}\}, \{x_{i+4},b\}$ give a $C_5$-minor in $(G,x_1x_2x_3x_4x_5)$, a contradiction.
        Therefore, for all $i\in [5]$, $x_ix_{i+1}\notin E(G)$ and $\{x_i,x_{i+1}\}\not\subseteq N(C)$ for every
        component $C$ of $G-(X\cup \{a,b\})$. 

        Let $\mathcal{C}$ denote the collection of all components of $G-(X\cup \{a,b\})$. If ${\cal C}=\emptyset$ then $|V(G)|=7$ and, since $\rho(V(G)\setminus X)\ge5|V(G)\setminus X|+1$, we have $ab,ax_i,bx_i\in E(G)$ for $i\in [5]$; so   the conclusion of Theorem~\ref{main} holds, contradicting (c). So ${\cal C}\ne \emptyset$.  For each $C\in \mathcal{C}$,  since $\{x_i,x_{i+1}\}\not\subseteq N(C)$ for all $i\in [5]$ (where $x_6=x_1$),   $N(C)\subseteq \{a,b,x_i,x_{i+2}\}$ for some $i\in [5]$ (where $x_7=x_2$). Thus,  $\rho(V(C))\le 5|V(C)|$ since $(G,X)$ is 5-massed. Hence, 
      \begin{align*} \rho(V(G)\setminus X)
           &=\sum_{C\in\mathcal{C}}\rho(V(C))+e(\{a,b\},X)+e(a,b)\\
            &\le 5\sum_{C\in\mathcal{C}}|V(C)|+11\\
            &= 5(|V(G)|-7)+11\\
           &=5|V(G)\setminus X|+1.
      \end{align*}
      Therefore, since $\rho(V(G)\setminus X)\ge 5|V(G)\setminus X|+1$ (as $(G,X)$ is 5-massed), we must have $\rho(V(G)\setminus X)=5|V(G)\setminus X|+1$. So $\rho(V(C))=5|V(C)|$ for all $C\in\mathcal{C}$, $e(\{a,b\},X)=10$, and $e(a,b)=1$. Thus, $G$ contains edges $ab, ax_i,bx_i$ for all $i\in[5]$. This shows the conclusion of Theorem~\ref{main} holds, contradicting (c). 
    \end{proof}

	\begin{lemma} \label{2.1}
		$(G,X)$ has no rigid separation of order at most $5$. 
	\end{lemma}
	\begin{proof}
		Suppose for a contradiction that $(A,B)$ is a rigid separation of $(G,X)$ of order at most $5$. We choose $(A,B)$ such that $|A\cap B|$ is minimal and, subject to this, $A$ is minimal. 

         Suppose $|A\cap B|\ge |X|$. If $G[A]$ has $|X|$ disjoint paths from $X$ to $A\cap B$, then $(G,X)$ is cycle-linked, a contradiction. So  $G[A]$ does not contain $|X|$ disjoint paths from $X$ to $A\cap B$. Then by Menger's theorem, $G[A]$ has a separation $(A',B')$ such that $|A'\cap B'|<|X|$, $X\subseteq A'$, and $A\cap B\subseteq B'$. Choose $(A',B')$ with $|A'\cap B'|$ minimum. Then $G[B']$ contains $|A'\cap B'|$ disjoint paths from $A'\cap B'$ to $A\cap B$. Thus $(A',B'\cup B)$ is a rigid separation  in $G$ of order less than $|A\cap B|$, contradicting the choice of $(A,B)$. 
         
         So we may assume $|A\cap B|\le |X|-1\le 4$. Consider the pair $(G'[A],X)$ where $G'$ is obtained from $G[A]$  by adding edges to make $G'[A\cap B]$ complete. Note that any cycle minor in $(G'[A],X)$  can be extended to a cycle minor in $(G,X)$ by replacing edges in $G'[A\cap B]$ with appropriate subgraphs of $G[B]$. Hence, $(G'[A], X)$ is not cycle-linked, since $(G,X)$ is not. 

         We now show that $(G'[A],X)$ is 5-massed.       Since $(G,X)$ is $5$-massed and $|A\cap B|<|X|$, we have $\rho(B\setminus A)\leq 5 |B\setminus A|$. 
        Hence, 
        \[\rho(V(G')\setminus X)\ge \rho(V(G)\setminus X)-\rho(B\setminus A)\ge (5(V(G)\setminus X)+1)-5|B\setminus A|=5|A\setminus X|+1.\]
        So $(G'[A],X)$ satisfies  \hyperref[M1]{(M1)}. Suppose $(G'[A],X)$ is not 5-massed; then it violates (M2). Hence, there is a separation $(A',B')$ of $(G'[A],X)$ of order less than $|X|$ such that $\rho_{G'[A]}(B'\setminus A')> 5 |B'\setminus A'|$. Choose such $(A',B')$ with $B'$ minimal.  Since $G'[A\cap B]$ is a clique, $A\cap B\subseteq A'$ or $A\cap B\subseteq B'$. If $A\cap B\subseteq A'$ then $(A'\cup B,B')$ is a separation of $(G,X)$ violating \hyperref[M2]{(M2)}, a contradiction. Thus we may assume $A\cap B\subseteq B'$. Consider the pair $(G'[B'],A'\cap B')$, which satisfies (M1) (as $\rho_{G'[A]}(B'-A')> 5 |B'-A'|$) and (M2) (by minimality of $B'$). Since $|A'\cap B'|<|X|$, $(G'[B'],A'\cap B')$ is cycle-linked (by the choice of $(G,X)$). Thus, $(G[B'\cup B],A'\cap B')$ is also cycle-linked; so the separation $(A',B'\cup B)$ is rigid in $(G,X)$, which contradicts the choice of $(A,B)$ (that is, the minimality of $A$).

        Since $(G'[A],X)$ is 5-massed but not cycle-linked, it follows from the choice of $(G,X)$ that $(G'[A],X)$ violates (c), i.e., the conclusion of Theorem~\ref{main} holds for $(G'[A],X)$. Thus, $|X|=5$, the vertices in $X$ may be labeled as $x_1,\ldots, x_5$ such that $x_ix_{i+1}\notin E(G)$ for $i\in [5]$, and there exist $a,b\in A\setminus X$ such that 
            \begin{itemize}
             \item [(1)] $ab\in E(G'[A])$ and $a,b\in N_{G'[A]}(x_i)$ for all $i\in [5]$, and
             \item [(2)] if $C$ is a component of $G'[A]-(X\cup \{a,b\})$ then $\rho_{G'[A]}(C)=5|V(C)|$ and $N_{G'[A]}(C)\subseteq \{a,b,x_i,x_{i+2}\}$ for some $i\in [5]$.
           \end{itemize}
           Since $G'[A\cap B]$ is a clique in $G'[A]$, either $A\cap B\subseteq N_{G'[A]}(C)\cup V(C)$ for some component $C$ of $G'-(X\cup \{a,b\})$, or $A\cap B\subseteq \{a,b,x_i,x_{i+2}\}$ for some $i\in[5]$. This implies that $(G,X)$ also satisfies (1) and (2) with the same choice of $a,b$ and ordering on $X$, contradicting (c).          
    \end{proof}

        We derive the following convenient fact from Lemma~\ref{2.1}

    \begin{corollary} \label{2.2}
        $G$ contains no $K_5$.
    \end{corollary} 

    \begin{proof}
        For, suppose that $G$ has a clique $H$ on $5$ vertices. If $G$ contains $|X|$ vertex disjoint paths from $X$ to $V(H)$, then clearly $(G,X)$ is cycle-linked. Hence, no such $|X|$ disjoint paths exist. Therefore, by Menger's theorem, $(G,X)$ has a separation $(A,B)$ such that $V(H)\subseteq B$ and $|A\cap B|<|X|$. Choose $(A,B)$ such that $|A\cap B|$ is minimum. By Menger's theorem again, $B$ contains $|A\cap B|$ disjoint paths from $A\cap B$ to $V(H)$, which shows that $(G[B],A\cap B)$ is cycle-linked. However, this means that $(A,B)$ is a rigid separation in $(G,X)$, contradicting Lemma~\ref{2.1}.  
    \end{proof}

    \begin{lemma}\label{5cut}
        If $|X|=5$, then for every separation $(A,B)$ of $(G,X)$ of order 5, we have $\rho(B\setminus A)\le 5|B\setminus A|+1$, and equality holds only when $(G[B],A\cap B)$ satisfies the conclusion of Theorem~\ref{main}.
    \end{lemma}

    \begin{proof}
        Suppose $(A,B)$ is a separation of $(G,X)$ of order 5. If $\rho(B\setminus A)\le 5|B\setminus A|$ then there is nothing to prove. So assume $\rho(B\setminus A)> 5|B\setminus A|$; thus $(G[B], A\cap B)$ satisfies \hyperlink{M1}{(M1)}. Note that $(G[B], A\cap B)$ also satisfies \hyperlink{M2}{(M2)} because $(G,X)$ satisfies \hyperlink{M2}{(M2)}. Now $(G[B], A\cap B)$ is not cycle-linked, since otherwise $(A,B)$ is a rigid separation in $(G,X)$, contradicting Lemma \ref{2.1}. Hence, by the choice of $(G,X)$, $(G[B],A\cap B)$ violates (c). Therefore, the conclusion of Theorem~\ref{main} holds for $(G[B],A\cap B)$ and, thus, the assertion of the lemma holds.
    \end{proof}

     To force a small dense subgraph in $G$, we consider $(G/uv,X)$ for all edges $uv$ with $\{u,v\}\not\subseteq X$. 
     
	\begin{lemma} \label{2.3}
		Let $uv\in E(G)$ with $\{u,v\}\not\subseteq X$. 
        If $|\{u,v\}\cap X|=0$ then $|N(u)\cap N(v)|\geq 5$. 
        If $|\{u,v\}\cap X|=1$ then $|(N(u)\cap N(v))\backslash X| + |N(\{u,v\}\setminus X)\cap X| \geq 6$.	
    \end{lemma}
	\begin{proof}	
        Let $G'=G/uv$, and let $w$ denote the vertex resulting from the contraction of $uv$. Then $(G',X)$ is not cycle-linked, since otherwise $(G,X)$ would also be cycle-linked. 
         
         \medskip
        {\it Claim} 1. We may assume that $(G',X)$ satisfies \hyperref[M1]{(M1)}. 
        
        For, otherwise, we have $\rho(V(G')\setminus X)\leq 5 |V(G')\setminus X|$. Therefore, since $\rho(V(G)\setminus X)>5|V(G)\setminus X|$, we have $$\rho(V(G)\setminus X) -\rho(V(G')\setminus X) > 5 |V(G)\setminus X| - 5 |V(G')\setminus X| =5.$$  
        If $|\{u,v\}\cap X|=0$ then $|N(u)\cap N(v)|=\rho(V(G)\setminus X) -\rho(V(G')\setminus X)-1\ge 5$.  If $|\{u,v\}\cap X|=1$ then $|(N(u)\cap N(v))\backslash X| + |N(\{u,v\}\setminus X)\cap X|  = \rho(V(G)\setminus X) -\rho(V(G')\setminus X)\geq 6$.

        \medskip
        {\it Claim} 2. We may assume that $(G',X)$ satisfies \hyperref[M2]{(M2)} and, hence, is 5-massed. 
        
		For, suppose that violates \hyperref[M2]{(M2)}. Then there is a separation $(A',B')$ of $(G',X)$ of order at most $|X|-1$ such that $\rho_{G'}(B'\setminus A')> 5 |B'\setminus A'|$. Choose such $(A',B')$ with $B'$ minimal. 
        Let $A=(A'\setminus \{w\})\cup \{u,v\}$ if $w\in A'$, and $A=A'$ otherwise; similarly, let $B=(B'\setminus \{w\})\cup \{u,v\}$ if $w\in B'$, and $B=B'$ otherwise. Then $(A,B)$ is a separation of $(G,X)$ of order $|A\cap B|\le |A'\cap B'|+1\le 5$. 
        
        Suppose $w\in B'\setminus A'$. Note that $(G'[B'], A'\cap B')$ satisfies \hyperref[M1]{(M1)} (by the choice of $(A',B')$) and \hyperref[M2]{(M2)} (by minimality of $B'$). 
        By the choice of $(G,X)$ and because $|A'\cap B'|<5$, $(G'[B'], A'\cap B')$ is cycle-linked. Hence,  $(G[B],A\cap B)$ is cycle-linked, since $A\cap B=A'\cap B'$ and $G'[B']=G[B]/uv$. Now $(A,B)$ is a rigid separation of order at most $4$ in $(G,X)$, contradicting Lemma \ref{2.1}.
        
        So $w\in A'$. Thus, $\rho_{G}(B\setminus A)=  \rho_{G'}(B'\setminus A')>5|B'\setminus A'|= 5|B\setminus A|$, i.e.,  $(G[B], A\cap B)$ satisfies \hyperref[M1]{(M1)}. Hence  $(G[B], A\cap B)$ is 5-massed, as it satisfies \hyperref[M2]{(M2)} (by the minimality of $B'$). Moreover, $(G[B], A\cap B)$ is not cycle-linked, since otherwise $(A,B)$ is a rigid separation in $(G,X)$ of order at most 5, contradicting Lemma \ref{2.1}.  Hence, by the choice of $(G,X)$, (c) fails for $(G[B], A\cap B)$, i.e., $(G[B], A\cap B)$ satisfies the conclusion of Theorem \ref{main}. In particular,  $|A\cap B|=5$ (so $|A'\cap B'|=4$),  
        $\rho_{G[B]}(B\setminus A)=5|B\setminus A|+1$, and $|N(u)\cap N(v)\cap (B\setminus A)|\ge 2$. So $\rho_{G'}(B'\setminus A')\le \rho_{G[B]}(B\setminus A)-2< 5|B\setminus A|=5|B'\setminus A'|$, which contradicts the choice of $(A',B')$   (that $\rho_{G'}(B'\setminus A')> 5|B'\setminus A'|$), completing the proof of Claim 2. 

    \medskip

        By Claim 2,  $(G',X)$ is $5$-massed but not cycle-linked. So by the choice of $(G,X)$,  (c) fails for $(G',X)$; that is, the conclusions of Theorem~\ref{main} holds for $(G',X)$.  Therefore, $|X|=5$, $\rho_{G'}(V(G')\setminus X)=5|V(G')\setminus X|+1$,  the vertices in $X$ may be labeled as $x_1,\ldots, x_5$ and there exist $a,b\in V(G')\setminus X$, such that 
        \begin{itemize}
            \item [(i)] $x_ix_{i+1}\notin E(G')$ for $i\in [5]$,     \item [(ii)] $ab\in E(G')$ and $a,b\in N_{G'}(x_i)$ for all $i\in [5]$, and
            \item [(iii)] if $C$ is a component of $G'-(X\cup \{a,b\})$ then $\rho_{G'}(V(C))=5|V(C)|$ and $N_{G'}(V(C))\subseteq \{a,b,x_i,x_{i+2}\}$ for some $i\in [5]$. 
        \end{itemize}
        
         By (iii), we have the following
         \medskip
         
         {\it Claim} 3. $V(G')\setminus (X\cup \{a,b\})$ can be partitioned into five (possibly empty) disjoint sets $V(G')\setminus (X\cup \{a,b\})=C_{1,3}\sqcup C_{2,4}\sqcup C_{3,5}\sqcup C_{4,1}\sqcup C_{5,2}$,  such that $N_{G'}(C_{i,i+2})\subseteq\{a,b,x_i,x_{i+2}\}$ for $i\in [5]$.
        
        \medskip
        {\it Claim} 4. $w\in X\cup \{a,b\}$. 
        
        For, suppose $w\notin X\cup \{a,b\}$. 
        Then $a,b\in\bigcap_{x\in X}N_G(x)$, a contradiction to Lemma \ref{exception}.
        
        \medskip

        We consider two cases based on Claim 4. 
        \medskip

        {\it Case} 1. $w\in X$. 
        
        Without loss of generality, we may assume that $u=x_1$ and $v\in V(G)\setminus X$. 
        Note that, by (iii),  $N_G(C_{i,i+2})=N_{G'}(C_{i,i+2})\subseteq\{a,b,x_i,x_{i+2}\}$ and $\rho_G(C_{i,i+2})=5|C_{i,i+2}|$ for $i\in\{2,3,5\}$, $N_G(C_{1,3})\subseteq\{a,b,x_1,x_3,v\}$, and $N_G(C_{4,1})\subseteq\{a,b,x_4,x_1,v\}$. 
        Also note that  $\{ab,vx_1\}\cup \{ax_i,bx_i: i\in\{2,3,4,5\}\}\subseteq E(G)$ (by (ii)) and $vx_2,vx_5\not\in E(G)$ (by (i)).  Moreover,  
        \begin{align*}
            &\quad \quad  \rho_G(V(G)\setminus X)-5|V(G)\setminus X|-1 \\
            &=\sum_{i\in[5]}\left(\rho_G(C_{i,i+2}\right)-5|C_{i,i+2}|)+e(\{v,a,b\},X)+|E(G[\{v,a,b\}])|-5\cdot 3-1\\
            &\le (\rho_G(C_{1,3})-5|C_{1,3}|)+(\rho_G(C_{4,1})-5|C_{4,1}|)+|E(G)\cap\{vx_3,vx_4,va,vb,x_1a,x_1b\}|-6\\
            &\le |E(G)\cap\{vx_3,vx_4,va,vb,x_1a,x_1b\}|-4,        \end{align*}        
        where the last inequality holds since $\rho(C_{i,i+2})-5|C_{i,i+2}|\le 1$ for $i\in\{1,4\}$ (by Lemma \ref{5cut}).  
        By       Lemma~\ref{exception}, we may assume $x_1\notin N(a)\cap N(b)$. Therefore, since  $\rho_G(V(G)\setminus X)-5|V(G)\setminus X|-1\ge 0$ (as $(G,X)$ satisfies (M2)),   we have  $$4\le |E(G)\cap\{vx_3,vx_4,va,vb,x_1a,x_1b\}|\le 5.$$

        Suppose $|E(G)\cap\{vx_3,vx_4,va,vb,x_1a,x_1b\}|=4$. Then we must have $\rho(C_{i,i+2})-5|C_{i,i+2}|=1$ for all $i\in \{1,4\}$, which, by Lemma \ref{5cut}, implies that $|N(u)\cap N(v)\cap C_{i,i+2}|\ge 2$ for $i\in \{1,4\}$;     
        so $|(N(u)\cap N(v))\setminus X|\ge 4$.  Since $|E(G)\cap \{vx_3,vx_4,va,vb,x_1a,x_1b\}|=4$ and $\{x_1a,x_1b\}\not\subseteq E(G)$, we have $E(G)\cap \{vx_3,vx_4\}\ne \emptyset$. Therefore, since $vx_1\in E(G)$, we have $|N(\{u,v\}\setminus X)\cap X| =|N(v)\cap X|\geq 2$. So	$|(N(u)\cap N(v))\backslash X| + |N(\{u,v\}\setminus X)\cap X| \geq 6$, as desired.	

        Now assume $|E(G)\cap\{vx_3,vx_4,va,vb,x_1a,x_1b\}|=5$. Without loss of generality, we may further assume $x_1a\not\in E(G)$. Then $vx_1,vx_3,vx_4\in E(G)$ (i.e., $|N(v)\cap X|\ge 3$), and $b\in N(u)\cap N(v)$.  Since $|E(G)\cap\{vx_3,vx_4,va,vb,x_1a,x_1b\}|=5$, there exists some $i\in \{1,4\}$ such that $\rho_G(C_{i,i+2})-5|C_{i,i+2}|=1$; hence by Lemma \ref{5cut}, $|N(u)\cap N(v)\cap C_{i,i+2}|\ge 2$.  So $|(N(u)\cap N(v))\backslash X| + |N(\{u,v\}\setminus X)\cap X| \geq (1+2)+3= 6$, as desired.

        \medskip
        {\it Case} 2.  $w\in\{a,b\}$. 
        
        Without loss of generality, let $w=b$. Then $X\subseteq N_{G'}(w)=N_G(u)\cup N_G(v)$ and $X\subseteq N_G(a)$. By Lemma~\ref{exception}, $|N_G(u)\cap X|\le 4$ and $|N_G(v)\cap X|\le 4$. 
        
        \medskip 
        {\it Subcase} 2.1. $|N_G(u)\cap X|=4$ or $|N_G(v)\cap X|=4$. 
        
         Without loss of generality, we may assume $N_G(u)\cap X=\{x_1,x_2,x_3,x_4\}$. So $x_5\in N(v)$. 
         
         We claim that for $i\in \{1,4\}$, $vx_i\notin E(G)$ and  $\{v,x_i\}\not\subseteq N(C)$ for every component $C$ of $G-(X\cup \{u,v,a\})$. For, otherwise, by the symmetry between $x_1$ and $x_4$, we may assume that $vx_4\in E(G)$ or $G-(X\cup \{u,v,a\})$ has a component $C$ with $\{v,x_4\}\subseteq N_G(C)$. Then let  $P$ be a $v$-$x_4$ path in $G[V(C)\cup \{v,x_4\}]$. Now   $\{x_1,a\}, \{x_2\}, \{x_3,u\}, V(P), \{x_5\}$ give a $C_5$-minor in $(G,x_1x_2x_3x_4 x_5)$, a contradiction.

        Therefore, we further choose the partition of $V(G')\setminus (X\cup \{a,b\})$ in Claim 3 to maximize $C_{3,5}\cup C_{5,2}$; so for each component $C$ of $G'[C_{1,3}]\cup  G'[C_{4,1}]\cup  G'[C_{2,4}]$, $N_{G'}(C)\cap \{x_1,x_4\}\ne \emptyset$.  Then by the above claim,  $N_G(C_{i,i+2})\subseteq \{u,a,x_i,x_{i+2}\}$ for $i\in \{1,2,4\}$. So $\rho_G(C_{i,i+2})\le 5|C_{i,i+2}|$ (by Claim 2) for $i\in \{1,2,4\}$.  Hence \begin{align*}
            &\quad\quad \rho_G(V(G)\setminus X)-5|V(G)\setminus X|-1\\
            &=\sum_{i\in[5]}(\rho_G(C_{i,i+2})-5|C_{i,i+2}|)+e(\{u,v,a\},X)+|E(G[\{u,v,a\}])|-5\cdot 3-1\\
            &\le (\rho_G(C_{3,5})-5|C_{3,5}|)+(\rho_G(C_{5,2})-5|C_{5,2}|)+|E(G)\cap\{vx_3,vx_2,ua,va\}|-5,
        \end{align*}        
        where the inequality holds since $\{ax_1,ax_2,ax_3,ax_4,ax_5,uv,ux_1,ux_2,ux_3,ux_4,vx_5\}\subseteq E(G)$ and $ux_5,vx_1,vx_4\not\in E(G)$. Since $(G,X)$ satisfies (M1), $\rho_G(V(G)\setminus X)-5|V(G)\setminus X|-1\ge 0$. For $i\in \{3,5\}$, if $|N_G(C_{i,i+2})|\le 4$ then $\rho_G(C_{i,i+2})-5|C_{i,i+2}|\le 0$ (by Claim 2); if $|N_G(C_{i,i+2})|=5$ then by Lemma \ref{5cut},  $\rho_G(C_{i,i+2})-5|C_{i,i+2}|\le 1$ with equality only when $(G[C_{i,i+2}\cup N_G(C_{i,i+2})], N_G(C_{i,i+2}))$ satisfies the conclusion of Theorem~\ref{main} (in particular, $|N_G(u)\cap N_G(v)\cap C_{i,i+2}|\ge 2$).  Hence,  $|E(G)\cap \{vx_3,vx_2,ua,va\}|\in \{3,4\}$. 
        
        First, suppose  $|E(G)\cap \{vx_3,vx_2,ua,va\}|=3$. Then, for both $i\in \{3,5\}$, $\rho_G(C_{i,i+2})-5|C_{i,i+2}|=1$ and  $|N_G(u)\cap N_G(v)\cap C_{i,i+2}|\ge 2$. Since  $N_G(u)\cap N_G(v)\cap \{x_3,x_2\}\ne \emptyset$ (as $|E(G)\cap \{vx_3,vx_2,ua,va\}|=3$),  we have $|N_G(u)\cap N_G(v)|\ge 5$ as desired. 
        
        Now suppose $|E(G)\cap \{vx_3,vx_2,ua,va\}|=4$.  Then $a,x_2,x_3\in N_G(u)\cap N_G(v)$ and, for some $i\in \{3,5\}$, $\rho_G(C_{i,i+2})-5|C_{i,i+2}|=1$ and  $|N_G(u)\cap N_G(v)\cap C_{i,i+2}|\ge 2$. Hence,   $|N_G(u)\cap N_G(v)|\ge 5$ as desired.

        \medskip
        {\it Subcase} 2.2.  $|N_G(u)\cap X|\le 3$ and $|N_G(v)\cap X|\le 3$. 

        Since $X\subseteq N_{G}(u)\cup N_G(v)$ and $|X|=5$, we may assume $|N_G(u)\cap X|=3$. If $N_G(u)=\{x_i,x_{i+1},x_{i+2}\}$ for some $i$, then $x_{i+3}, x_{i+4}\in N_G(v)$ and, hence, $\{x_i,a\}, \{x_{i+1}\}$, $\{x_{i+2},u\}$, $\{x_{i+3},v\}$, $\{x_{i+4}\}$ give a $C_5$-minor in $(G,x_1x_2x_3x_4x_5)$, a contradiction. So, by symmetry, we may assume $N_G(u)=\{x_1,x_2,x_4\}$ and, hence, $x_3,x_5\in N(v)$. 

        We claim that $vx_4\notin E(G)$ and $\{v,x_4\}\not \subseteq N_G(C)$ for every component $C$ of $G-(X\cup \{u,v,a\})$. Otherwise, let $P$ be a $v$-$x_4$ path in $G[V(C)\cup \{v_1,x_4\}]$. Now $\{x_1\}, \{x_2,u\}, \{x_3,v\}$, $V(P-v), \{x_5,a\}$ give an ordered $C_5$-minor in $(G,x_1x_2x_3x_4x_5)$, a contradiction.

        Suppose $G-(X\cup \{u,v,a\})$ has a component $C$ such that $\{x_3,x_5,u\}\subseteq N_G(C)$. Then, for $i\in [2]$,  $vx_i\not\in E(G)$ and $\{v,x_i\}\not\subseteq N_G(D)$ for all  components $D$ of $G-(X\cup \{u,v,a\})$ with $D\ne C$;  since otherwise, $G[V(D)\cup \{v,x_i\}]$ contains a $v$-$x_i$ path, say $Q$, and  $V(Q), \{x_{2},u\}, \{x_{3}\}\cup V(C), \{x_{4},a\},\{x_{5}\}$ (when $Q$ is a $v$-$x_1$ path) or  $\{x_{1},u\},V(Q), \{x_3\}, \{x_{4},a\}, \{x_{5}\}\cup V(C)$ (when $Q$ is a $v$-$x_2$ path) would give a $C_5$-minor in $(G,x_1x_2x_3x_4x_5)$, a contradiction. We choose the partition $V(G')\setminus (X\cup \{a,b\})$ in Claim 3 to first minimize $C_{2,4}\cup C_{4,1}$ and then maximize 
        $C_{3,5}$.  Then $ N_G(C_{2,4})\subseteq \{u,a,x_2,x_4\}$, $N_G(C_{4,1})\subseteq \{u,a,x_1,x_4\}$, $N_G(C_{1,3})\subseteq \{x_1,x_3,u,a\}$, and $N_G(C_{5,2})\subseteq\{x_2,x_5,u,a\}$. Thus,  $N_G(C_{3,5}\cup\{v\})\subseteq\{u,a,x_3,x_5\}$. Now, since $(G,X)$ satisfies (M2), $\rho_G(C_{i,i+2})\le 5|C_{i,i+2}|$ for $i=[5]\setminus \{3\}$ and $\rho_G(C_{3,5}\cup \{v\})\le 5|C_{3,5}\cup \{v\}|$. Hence, 
        \begin{align*}
            \rho_G(V(G)\setminus X)&=\sum_{i\in [5]\setminus \{3\}}\rho_G(C_{i,i+2}) +\rho_G(C_{3,5}\cup \{v\})+e(\{a,u\},X)+e(u,a)\\
            &\le 5|V(G)\setminus (X\cup \{u,a\})|+8+1\\
            &=5|V(G)\setminus X|-1.
        \end{align*}
        But this shows that $(G,X)$ does not satisfy (M1), a contradiction.

        Therefore, we may assume that $\{x_3,x_5,u\}\not\subseteq N_G(C)$ for all components $C$ of $G-(X\cup \{u,v,a\})$. Thus, we may choose the partition $V(G')\setminus (X\cup \{a,b\})$ in Claim 3 to minimize $C_{2,4}\cup C_{4,1}\cup C_{3,5}$. Then $N_G(C_{2,4})\subseteq \{u,a,x_2,x_4\}$, $N_G(C_{4,1})\subseteq \{u,a,x_1,x_4\}$, and $N_G(C_{3,5})\subseteq \{v,a,x_3,x_5\}$. Hence, $\rho_G(C_{i,i+2})\le 5|C_{i,i+2}|$ for $i=2,3,4$. Thus, 
        \begin{align*}
             &\quad\quad \rho_G(V(G)\setminus X)-5|V(G)\setminus X|-1\\
             &=\sum_{i\in[5]}(\rho_G(C_{i,i+2})-5|C_{i,i+2}|)+e_G(\{u,v,a\},X)+|E(G[\{u,v,a\}])|-5\cdot 3-1\\
            &\le (\rho_G(C_{1,3})-5|C_{1,3}|)+(\rho_G(C_{5,2})-5|C_{5,2}|)+|E(G)\cap\{vx_1,vx_2,ua,va\}|-5, 
        \end{align*}        
        since $\{ax_1,ax_2,ax_3,ax_4,ax_5,uv,ux_1,ux_2,ux_4,vx_3,vx_5\}\subseteq E(G)$ and $ux_3,ux_5,vx_4\not\in E(G)$. Since $(G,X)$ satisfies (M1), $\rho_G(V(G)\setminus X)-5|V(G)\setminus X|-1\ge 0$. Hence, by Lemma \ref{5cut},  $|E(G)\cap \{vx_1,vx_2,ua,va\}|\in \{3, 4\}$.  
        
        Suppose $|E(G)\cap \{vx_1,vx_2,ua,va\}|=3$. Then $\rho_G(C_{1,3})-5|C_{1,3}|=\rho_G(C_{5,2})-5|C_{5,2}|=1$ and $|N_G(u)\cap N_G(v)\cap \{x_1,x_2\}|\ge 1$. By Lemma~\ref{5cut}, $|N_G(u)\cap N_G(v)\cap C_{1,3}|\ge 2$ and $|N_G(u)\cap N_G(v)\cap C_{5,2}|\ge 2$.   Thus,  $|N_G(u)\cap N_G(v)|\ge 5$, as desired.
        
        So assume $|E(G)\cap \{vx_1,vx_2,ua,va\}|=4$.  Then $a,x_1,x_2\in N_G(u)\cap N_G(v)$, and $\rho_G(C_{1,3})-5|C_{1,3}|=1$ or $\rho_G(C_{5,2})-5|C_{5,2}|=1$. So by Lemma~\ref{5cut},         $|N_G(u)\cap N_G(v)\cap C_{3,5}|\ge 2$ or $|N_G(u)\cap N_G(v)\cap C_{5,2}|\ge 2$. Again, $|N_G(u)\cap N_G(v)|\ge 5$, as desired.
	\end{proof} 
	
	\begin{lemma} \label{2.4}
		There exists a vertex $x^*\in V(G)\backslash X$ with $|N(x^*)|<10$.
	\end{lemma}
	\begin{proof}
        First, we claim that $\rho(V(G)\setminus X)\le 5|V(G)\setminus X|+2$. For, otherwise, let $uv\in E(G)\setminus E(G[X])$ and $G'=G-uv$; then  $\rho_{G'}(V(G')\setminus X)=\rho_G(V(G)\setminus  X)-1\ge  5 |V(G)\setminus X|+3-1=5|V(G')\setminus X|+2$.  Thus, $(G', X)$ does not satisfy the conclusion of Theorem \ref{main}. Note that $(G',X)$ is not cycle-linked as $(G,X)$ is not. Therefore, by the choice of $(G,X)$, $(G',X)$ is not 5-massed; and since $(G',X)$ satisfies (M1), it violates (M2). So, there exists a separation $(A,B)$ of $(G',X)$ of order at most $|X|-1$, such that $\rho_{G'}(B\setminus A)> 5|B\setminus A|$. If $\{u,v\}\subseteq A$ or $\{u,v\}\subseteq  B$, then $(A,B)$ is also a separation of $(G,X)$ and $\rho_{G}(B\setminus A)> 5|B\setminus A|$, contradicting the assumption that $(G,X)$ satisfies (M2). Thus, we may assume that $u\in A\backslash B$ and $v\in B\backslash A$. Let $A':=A$ and $B':=B\cup \{u\}$. Then $(A',B')$ is a separation of $(G,X)$ of order at most $|X|$. Now consider the pair $(G'[B'],A'\cap B')$. 
		Note that $\rho_{G[B']}(B'\setminus A') = \rho_{G'}(B\setminus A) +1 \geq 5 |B\setminus A|+2=5 |B'\setminus A'|+2$, so $(G'[B'],A'\cap B')$ satisfies (M1) but does not satisfy the conclusion of Theorem \ref{main}. Moreover,  for each separation $(A'',B'')$ of $(G'[B'],A'\cap B')$ of order less than $|A'\cap B'|$, $(A''\cup A,B'')$ is a separation of $(G,X)$ of the same order; so $\rho_{G'[B']}(B''\setminus A'')=\rho_G(B''\setminus (A''\cup A))\le 5|B''\setminus (A''\cup A)| =5|B''\setminus A''|$. Thus  $(G'[B'],A'\cap B')$ satisfies (M2) and hence is $5$-massed. By our choice of $(G,X)$, $(G'[B'],A'\cap B')$ is cycle-linked. Thus, $(A',B')$ is a rigid separation in $(G,X)$, a contradiction to Lemma \ref{2.1}.
        
        Therefore, we have $$\sum_{v\in V(G)} |N(v)| = 2|E(G)| =2\rho(V(G)\setminus X) +2|E(G[X])| \leq 10|V(G)\setminus X| +4 +2|E(G[X])|.$$
        
        Now, suppose for a contradiction that the assertion of Lemma \ref{2.4} is false. Then, for every vertex $v\in V(G)\backslash X$, we have $|N(v)|\geq 10$. So $$\sum_{v\in V(G)} |N(v)| = \sum_{v\in V(G)-X} |N(v)| +\sum_{x\in X} |N(x)|\geq 10 |V(G)\setminus X| +\sum_{x\in X} |N(x)|.$$       
        Combining the above two inequalities, we have $$ 2|E(G[X])|+ |E(X,V(G)\setminus X)| = \sum_{x\in X} |N(x)| \leq 4 +2|E(G[X])|.$$	
		Thus, $|E(X,V(G)\setminus X)|\leq 4$, which means that there exists a vertex $x\in X$ such that $N(x)\cap (V(G)\setminus X) = \emptyset$. But then $(X,V(G)\backslash \{x\})$ is a separation of $(G,X)$ violating (M2), a contradiction.
	\end{proof}

    \section{Small dense subgraphs} \label{sec3}

In this section, we complete the proof of Theorem~\ref{main}. Suppose Theorem \ref{main} does not hold. Then there exists a graph $G$ and a set $X\subseteq V(G)$ such that $(G,X)$ is 5-massed but not cycle-linked, the conclusion of Theorem~\ref{main} fails, and,  subject to these conditions, $|V(G)|+|E(G)|$ is minimal.

 By Corollary~\ref{2.3} we have $K_5\not\subseteq G$, and by Lemma~\ref{2.4} there exists a vertex $a\in V(G)\setminus X$ such that $|N(a)|<10$. Let $H:=G[N[a]]$. 
\begin{lemma}\label{3.0}
    $K_4\not\subseteq H-a$ and $7\le |V(H)|\le 10$, if $v\in V(H)\setminus X$ then $|N_H(v)|\ge 6$, and if $x\in V(H)\cap X$ then $|N_H(x)\setminus X|\ge 2$ and $|N_H(x)\setminus X|+|V(H)\cap X|\ge 7$.
\end{lemma}
\begin{proof}
     Since $K_5\not\subseteq G$, $K_4\not\subseteq H-a$. Since $|N_G(a)|<10$, $|V(H)|\le 10$. 
     Note that $|V(H)\setminus X|\ge 2$, since otherwise, $N_G(a)\subseteq X$ and $(G-a,X)$ contradicts the choice of $(G,X)$. 
     
    Suppose $x\in V(H)\cap X$. Then $(|N_H(x)\setminus X|-1)+|V(H)\cap X|=|(N_G(a)\cap N_G(x))\backslash X| + |N(\{a,x\}\setminus X)\cap X| \geq 6$ (by Lemma \ref{2.3}). This implies $|N_H(x)\setminus X|\ge 2$ and $|N_H(x)\setminus X|+|V(H)\cap X|\ge 7$. 
    
    Now suppose $v\in V(H)\setminus (X\cup\{a\})$ (which exists, since $N_G(a)\not\subseteq X$). By Lemma~\ref{2.3}, 
     we have $|N_H(v)|\ge |N_G(a)\cap N_G(v)|+|\{a\}|\ge 5+1=6$ and $|N_H(a)|\ge |N_G(a)\cap N(v)|+|\{v\}|\ge 5+1=6$. Hence,  $|V(H)|\ge 7$.
\end{proof}

   	\begin{lemma} \label{3.1}
		For every $X'\subseteq V(H)\setminus \{a\}$ with $2\le |X'|\le 4$ and $X\cap V(H)\subseteq X'$, $(H,X')$ is cycle-linked.
	\end{lemma}
	\begin{proof} 
		Fix an arbitrary ordering of the vertices in $X'$, say $x_1',\ldots, x_t'$, where $t=|X'|\in \{2,3,4\}$. We will show that $(H, x_1'\ldots x_t')$ has a $C_t$-minor if $t\in\{3,4\}$, and $H$ contains a $x_1'$-$x_2'$ path if $t=2$.  Let $U=V(H)\setminus (X'\cup \{a\})$.
   
        If $t=2$ then $x_1,a,x_2$ give the desired path in $(H, X')$. 
        
        Now suppose $t=3$. If there is a permutation $ijk$ of $[3]$ such that $H-\{a,x_k'\}$ contains an $x_i'$-$x_j'$ path $P$, 
        then $\{V(P)\setminus \{x_j'\}, \{x_j'\}, \{x_k',a\}$ give the desired $C_3$-minor in $(H, x_1'x_2' x_3')$.  So we may assume $x_1',x_2',x_3'$ belong to three different components of $H-a$, say $C_1,C_2,C_3$, respectively. Now by Lemma~\ref{3.0},  for each $i\in[3]$, $|N_{C_i}(x_i')|=|N_H(x_i')\setminus\{a\}|\ge 2-1=1$; so let $y_i\in N_{C_i}(x_i')$. By Lemma~\ref{3.0} again, $|N_{C_i}(y_i)|=|N_H(y_i)\setminus\{a\}|\ge 6-1=5$, which implies $|V(C_i)|\ge 6$. But then $|V(H)|>|C_1|+|C_2|+|C_3|\ge 18,$ a contradiction.

        Therefore, we may assume $t=4$. Then $|U|\le 5$ since $|V(H)|\le 10$. If there exists $u\in U$ such that $|N_H(u)\cap X'|\ge 3$ then we may let  $x_1',x_{2}',x_{3}'\in N(u)\cap X'$; now   $x_1',\{x_{2}',u\}, \{x_{3}'\},\{x_{4}',a\}$ give the desired $C_4$-minor in $(H,x_1'x_2'x_3'x_4')$. So we may assume that $|N_H(u)\cap X'|\le 2$ for all $u\in U$. Then $|N_H(u)\cap U|\ge 3$ for all $u\in U$ by Lemma~\ref{3.0}. Hence, the subgraph of $H$ induced by $U$ has minimum degree at least 3; so it must be connected as it has at most 5 vertices. Note that for each $i\in [4]$, $N(x_i')\cap U\ne \emptyset$ by Lemma~\ref{3.0}. Hence, $\{x_1'\},\{x_2',a\},\{x_3'\},\{x_4'\}\cup U$ give the desired $C_4$-minor in $(H,x_1'x_2'x_3'x_4')$. 
    \end{proof}

    \begin{lemma}\label{3.2}
		For every subset $X'\subseteq V(H)\setminus \{a\}$ such that $|X'|=5$ and $X\cap V(H)\subseteq X'$. $(H,X')$ is cycle-linked.
	\end{lemma}
    \begin{proof} 
    	Fix an arbitrary ordering of the  vertices in $X'$, say $x_1',x_2',\ldots,x_5'$. We will show that $(H,x_1'x_2'x_3'x_4'x_5')$ has a $C_5$-minor.     Set $U:=V(H)\setminus (X'\cup \{a\})$. Since $7\le |V(H)|\le 10$, we have $1\le |U|\le 4$.  
        \medskip
        
	    {\it Case} 1. $|U|=1$. 
        
        Let $u\in U$. Then, by Lemma~\ref{3.0}, $X'\subseteq N_H(u)$. If, for some $i\in [5]$, $x_i'x_{i+1}'\in E(G)$, then  $\{x_{i-2}'\},\{x_{i-1}',a\}, \{x_i'\}, \{x_{i+1}'\},\{x_{i+2}',u\}$ give a $C_5$-minor in $(H,x_1'x_2'x_3'x_4'x_5')$. So we may assume $x_i'x_{i+1}'\notin E(G)$ for all $i\in [5]$. Then $X'=X$;  for, otherwise, there exists $i\in [5]$ such that $x_i'\in X'\setminus X$ and, hence, $|N_H(x_i')|\le 4$, contradicting Lemma~\ref{3.0}. 

        Therefore, $|X|=|X'|=5$, and $a,u\in \bigcap_{x\in X}N(x)$, contradicting Lemma \ref{exception}.

        \medskip 

        {\it Case 2}.  $|U|\ge 2$. 
        
        Let $u_1\in U$ with $d_1:=|N_H(u_1)\cap U|$ minimum, and let $u_2\in U\setminus \{u_1\}$ with $d_2:=|N_H(u_2)\cap U|$ minimum.  
	
	    Suppose $d_2\le 1$. Then, by Lemma~\ref{3.0}, $|N_H(u_i)\cap X'|\ge 4$ for $i\in [2]$. Without loss of generality, assume that $\{x_1',x_2',x_3',x_4'\}\subseteq N_H(u_1)$. Note that $|N_H(u_2)\cap \{x_1',x_2',x_3',x_4'\}|\ge 3$; so  we may further assume without loss of generality that $x_3',x_4'\in N_H(u_2)$. Now $\{x_1'\}$, $\{x_2',u_1\}$, $\{x_3',u_2\}$,$\{x_4'\},\{x_5',a\}$ give a $C_5$-minor in $(H,x_1'x_2'x_3'x_4'x_5')$. 

    	Hence, we may assume $d_2\ge 2$. This implies that $U\setminus \{u_1\}$ induces a connected subgraph of $H$. Indeed, $d_2=2$, since $K_4\not\subseteq H-a$ by Lemma~\ref{3.0}. So for $i\in [2]$, we have $|N_H(u_i)\cap X'|\ge 3$. 
    
    	\medskip 
        {\it Subcase} 2.1.  $|N_H(u_1)\cap X'|\ge 4$. 
        
        Without loss of generality, let $x_1',x_2',x_3',x_4'\in N_H(u_1)$. If $x_1',x_2'\in N_H(u_2)$ or $x_3',x_4'\in N_H(u_2)$ then, by symmetry, let $x_3',x_4'\in N(u_2)$; now $\{x_1'\}$, $\{x_2',u_1\}$, $\{x_3',u_2\}$,$\{x_4'\},\{x_5',a\}$ give the desired $C_5$-minor in $(H,x_1'x_2'x_3'x_4'x_5')$. If $x_1',x_5'\in N_H(u_2)$ or $x_4',x_5'\in N_H(u_2)$ then, by symmetry, let $x_1',x_5'\in N_H(u_2)$; now  $\{x_1',u_2\}$, $\{x_2',u_1\}$, $\{x_3'\}$,$\{x_4',a\},\{x_5'\}$ give the desired $C_5$-minor in $(H,x_1'x_2'x_3'x_4'x_5')$. 
        Hence, we may assume that $N_H(u_2)=\{x_2',x_3',x_5'\}$. Now if $x_5'\in N_H(u_1)$, then  $\{x_1,a\}, \{x_2\}, \{x_3,u_2\},\{x_4,u_1\},\{x_5\}$ give the desired $C_5$-minor; so, we may assume $x_5'\not\in N_H(u_1)$, and hence $d_1\ge 1$ by Lemma \ref{3.0}.

        Suppose there exists $u_3\in U\setminus \{u_1,u_2\}$ such that $u_3\notin N_H(u_1)$. Then $|N_H(u_3)\cap U|=2$ and the argument above showing $N_H(u_2)=\{x_2',x_3',x_5'\}$ also applies to $u_3$. Hence, we may also assume  $N_H(u_3)=\{x_2',x_3',x_5'\}$. 
        Since $d_1\ge 1$ and $d_2=2$, there exists $i\in \{2,3\}$ such that $U-\{u_i\}$ induces a connected subgraph of $H$. Now $\{x_1',u_1\}, \{x_2',u_i\}, \{x_3'\}, \{x_4',a\},\{x_5'\}\cup (U\setminus \{u_1,u_i\})$ give the desired $C_5$-minor in $(H,x_1'x_2'x_3'x_4'x_5')$.  

        Hence, we may assume $U\setminus \{u_1,u_2\}\subseteq  N_H(u_1)$. 
        If $x_5'$ has a neighbor in $U\setminus \{u_1,u_2\}$, then $\{x_1',u_1\}, \{x_2',u_2\}, \{x_3'\}, \{x_4',a\},\{x_5'\}\cup (U\setminus \{u_1,u_2\})$ give the desired $C_5$-minor in $(H, x_1'x_2'x_3'x_4'x_5')$. Otherwise, every vertex in $U\setminus \{u_1,u_2\}$ has at least 3 neighbors in $U$. So $|U|=4$ and $H[U]$ has edge set $\{u_1u_3,u_1u_4,u_2u_3,u_2u_4,u_3u_4\}$. If $x_1'\in N(u_3)$, then $\{x_1',u_3\}, \{x_2',u_2\}, \{x_3'\}, \{x_4',u_1\},\{x_5',a \}$ give the desired $C_5$-minor. If $x_4'\in N(u_3)$, then  $\{x_1',u_1\}, \{x_2'\}, \{x_3',u_2\}, \{x_4',u_3\},\{x_5',a \}$ give the desired $C_5$-minor. So we may assume $N_H(u_3)=N_H(u_4)=\{x_2',x_3'\}$. Now $\{x_1',a\}, \{x_2'\}$, $\{x_3', u_3\}, \{x_4',u_1\},\{x_5',u_2,u_4\}$ give the desired $C_5$-minor.

    	\medskip
        
        {\it Subcase} 2.2. $|N_H(u_1)\cap X'|=3$. 
        
        Then $d_1=2$. Hence, since $K_4\not\subseteq H-a$ and $|U|\le 4$, we may further choose $u_1,u_2$ such that $U\backslash \{u_1,u_2\}\subseteq N_H(u_1)\cap N_H(u_2)$. Then, for each $i\in [2]$, $U\setminus \{u_i\}$ induces a connected subgraph in $H$. If for some $i\in [5]$ we have $x_{i+1}',x_{i+2}',x_{i+3}'\in N_H(u_1)$ then $x_{i+3}',x_{i+4}'\in N_H(U\setminus \{u_1\})$;  hence, $\{x_i',a\}, \{x_{i+1}'\}, \{x_{i+2}',u_1\},\{x_{i+3}'\}\cup (U\setminus \{u_1\}), \{x_{i+4}'\}$ give the desired $C_5$-minor in $(H,x_1'x_2'x_3'x_4'x_5')$. Therefore, we may assume that such $i$ does not exist. Thus, there exist $i,j\in [5]$ such that $N_H(u_1)\cap X'=\{x_i',x_{i+1}',x_{i+3}'\}$ and $N_H(u_2)\cap X'=\{x_j',x_{j+1}',x_{j+3}'\}$. Without loss of generality, we may assume that $i=1$. 
	    
	    If $j=1$ then $\{x_3',x_5'\}\subseteq N_H(U\setminus \{u_1,u_2\})$; now $\{x_1',u_1\}, \{x_2'\},\{x_3',a\},\{x_4'\},\{x_5'\}\cup (U\setminus \{u_1\})$ give  the desired $C_5$-minor $(H,x_1'x_2'x_3'x_4'x_5')$.  If $j=3$ then $\{x_1'\}, \{x_2',u_1\}, \{x_3'\}\cup (U\backslash \{u_1\}), \{x_4'\}, \{x_5',a\}$ give the desired $C_5$-minor in $(H,x_1'x_2'x_3'x_4'x_5')$. If $j=4$ then  $\{x_1',u_1\}$, $\{x_2'\}, \{x_3',a\}, \{x_4'\}, \{x_5'\}\cup (U\backslash \{u_1\})$ give the desired $C_5$-minor $(H,x_1'x_2'x_3'x_4'x_5')$. 

	    Now by symmetry we may assume that $j=2$. Let $u_3\in U\setminus \{u_1,u_2\}$. If $x_5'\in N_H(u_3)$ then $\{x_1',u_1\}, \{x_2',u_2\}$, $\{x_3'\}, \{x_4',a\}, \{x_5',u_3\}$ give the desired $C_5$-minor in $(H,x_1'x_2'x_3'x_4'x_5')$. If $x_4'\in N_H(u_3)$ then $\{x_1',u_1\}$, $\{x_2'\}$, $\{x_3',u_2\}$, $\{x_4',u_3\}$, $\{x_5',a\}$ give the desired $C_5$-minor in  $(H,x_1'x_2'x_3'x_4'x_5')$. So we may assume $x_4',x_5'\notin N_H(u_3)$. 
        
	    Suppose  $u_2\notin N_H(u_1)$. Then let $u_4\in U\setminus \{u_1,u_2,u_3\}$ and we know $u_4\in N_H(u_1)\cap N_H(u_2)$ and we may assume $x_4',x_5'\notin N_H(u_4)$. 
	    Note that $u_3u_4\notin E(H)$; since otherwise $\{u_2,u_3,u_4\}$ and one of the vertices in $\{x_2',x_3'\}$ induce a $K_4$ in $H-a$, contradicting Lemma~\ref{3.0}.        Hence, $N_H(u_3)\cap X' = N_H(u_4)\cap X'=\{x_1',x_2',x_3'\}$, which means $\{x_1'\}, \{x_2',u_3\}$, $\{x_3',u_4\}, \{x_4',u_1\}, \{x_5',a\}$ give the desired $C_5$-minor $(H,x_1'x_2'x_3'x_4'x_5')$.  
        
        Hence, we may assume that $u_2\in N_H(u_1)$. Then $U$ induces a triangle $u_1u_2u_3u_1$ in $H$. In particular, we also have $|N_H(u_3)\cap X'|=3$; so $x_1',x_2',x_3'\in N_H(u_3)$. But then $(H,x_1'x_2'x_3'x_4'x_5')$ has the desired $C_5$-minor $\{x_1'\}, \{x_2',u_3\}$, $\{x_3',u_2\}, \{x_4',u_1\}, \{x_5',a\}$. 
    \end{proof}

        We can now conclude the proof of Theorem~\ref{main}. Let $t=|X|\le 5$. Since $(G,X)$ is not cycle-linked, we may let $x_1,\ldots,x_{t}$ be an ordering of the vertices in $X$, such that $(G,x_1\ldots x_{t})$ has no $C_{t}$-minor. 
        
        First, suppose $G$ has $t$ disjoint paths $P_i$, $i\in [t]$, from $x_i$ to $V(H)$ and internally disjoint from $H$. For each $i\in [t]$, let $x_i'\in V(H)$ be the end of $P_i$. Since $N_G(a)\subseteq V(H)$, we see that $a\notin V(P_i)$ for $i\in [t]$. By Lemma~\ref{3.1} and  Lemma~\ref{3.2}, $(H,x_1'\ldots x_t')$ has a $C_t$-minor which, combined with $P_i$, $i\in [t]$, gives a $C_t$-minor in $(G,x_1\ldots x_t)$, a contradiction.

        Therefore, we may assume that such $t$ paths do not exist. Then by Menger's theorem, there is a separation $(A,B)$ of $(G,X)$ such that $V(H)\subseteq B$ and $|A\cap B|<t\le 5$. Choose $(A,B)$ to minimize the order $|A\cap B|$. Then by Menger's theorem again, $G[B]$ has $|A\cap B|$ disjoint paths from $A\cap B$ to $H$ and internally disjoint from $H$. Let $X'$ be the end of those paths in $H$. Then $|X'|\le 4$ and, by Lemma~\ref{3.1},  $(H,X')$ is cycle-linked. Hence, $(G[B], A\cap B)$ is also cycle-linked. This implies that $(A,B)$ is a rigid separation of $(G,X)$ of order at most 4, contradicting Lemma~\ref{2.1}.
        \qed
        
	\newpage
    
	\bibliographystyle{abbrv}
	\bibliography{reference}
\end{document}